\documentclass[draft, reqno]{amsart}
\usepackage[all]{xy}                        %

\CompileMatrices                            % Faster

\UseTips                                    % Use

\input xypic
\usepackage[bookmarks=true]{hyperref}       % Hyperref
\usepackage{amssymb,latexsym,amsmath,amscd}
\usepackage{xspace}
\usepackage{enumerate}
\usepackage{graphicx}
\usepackage{vmargin}
\usepackage{todonotes}

\DeclareMathOperator{\red}{red}

\DeclareMathOperator{\reg}{reg}

\reversemarginpar

\theoremstyle{plain}
\newtheorem{theorem}{Theorem}[section]
\newtheorem*{theorem*}{Theorem}

\newtheorem{corollary}[theorem]{Corollary}
\newtheorem{lemma}[theorem]{Lemma}

\newtheorem{question}[theorem]{Question}

\theoremstyle{definition}

\newtheorem{remark}[theorem]{Remark}

\newcommand{\enm}[1]{\ensuremath{#1}}          %
\newcommand{\cal}[1]{\mathcal{#1}}

\newcommand{\CC}{\enm{\mathbb{C}}}

\newcommand{\PP}{\enm{\mathbb{P}}}

\newcommand{\Zz}{\enm{\cal{Z}}}

\renewcommand{\phi}{\varphi}
\renewcommand{\theta}{\vartheta}
\renewcommand{\epsilon}{\varepsilon}

\begin{document}

\title[secant varieties]
{The $b$-secant variety of a smooth curve has a codimension $1$ locally closed subset whose points have rank at least $b+1$}
\author{E. Ballico}
\address{Dept. of Mathematics\\
 University of Trento\\
38123 Povo (TN), Italy}
\email{ballico@science.unitn.it}
\thanks{The author was partially supported by MIUR and GNSAGA of INdAM (Italy).}

\begin{abstract}
Take a smooth, connected and non-degenerate projective curve $X\subset \PP^r$, $r\ge 2b+2\ge 6$, defined over an algebraically closed
field with characteristic $0$ and let
$\sigma _b(X)$ be the
$b$-secant variety of $X$. We prove that the $X$-rank of $q$ is at least $b+1$ for a non-empty codimension $1$ locally closed
subset of $\sigma _b(X)$.
\end{abstract}

\maketitle

\section{Introduction}
Let $X\subset \PP^r$ be an integral and non-degenerate projective variety defined over an algebraically closed field. For any $q\in X$ the $X$-rank $r_X(q)$ of $X$
is the minimal cardinality of a set $S\subset X$ such that $q\in \langle S\rangle$, where $\langle \ \rangle$ denotes the linear span. For any integer
$s>0$ let $\sigma _s(X)\subseteq \PP^r$ be the $s$-secant variety of $X$, i.e. the closure of the union of all linear spaces
$\langle S\rangle$ with $S\subset X$ and $\sharp (S)=s$. See \cite{l} for many applications of $X$-ranks (e.g. the tensor rank) and secant varieties (a.k.a. the border rank). The algebraic set $\sigma _s(X)$ is an integral projective variety
of dimension $\le s(1+ \dim X) -1$ and $\sigma _s(X)$ is said to be non-defective if it has dimension $\min \{r,s(1+\dim X)-1\}$. Every secant variety of a curve
is non-defective (\cite[Corollary 1.4]{a}). Let $\tau (X)\subseteq \PP^r$ be the tangential variety of $X$, i.e. the closure in $\PP^r$ of the union of all tangent spaces
$T_pX$, $p\in X_{\reg}$. The algebraic set $\tau (X)$ is an integral projective variety of dimension $\le 2(\dim X)$ and $\tau (X)\subseteq \sigma _2(X)$. For any integer $b\ge 2$ let $\tau (X,b)$ denote the join of one copy of $\tau (X)$ and $b-2$ copies of $X$. If $X$ is a curve, then $\dim \tau (X,b) = \min \{r,2b-2\}$ (use $b-2$ times \cite[part 2) of Proposition 1.3]{a}
and that $\dim \tau (X)=2$) and hence $\tau (X,b)$ is a non-empty codimension $1$ subset of $\sigma _b(X)$ if $X$ is a curve and $r> 2b$. For a projective variety $X$ of arbitrary dimensional
usually $\tau (X,b)$ is a hypersurface of $\sigma _b(X)$, but this is not always true. For instance, if $\sigma _b(X)$ has not the expected dimension one expects that
$\tau (X,b) = \sigma _b(X)$ and this is the case if $X$ is smooth (\cite[Corollary 4]{fh}).

\begin{question}\label{z1}
Assume $b\ge 2$, $r\ge b(1+\dim X)-2$, and that $\sigma _b(X)$ has the expected dimension. Is $r_X(q)>b$ for a non-empty locally closed subset of $\sigma _b(X)$ of codimension
$1$ in $\sigma _b(X)$? Is $r_X(q)>b$ for a general point of $\tau (X,b)$?
\end{question}

In this note we prove the following result.

\begin{theorem}\label{i1}
Fix an integer $b\ge 2$ and let $X\subset \PP^r$, $r\ge 2b+2$, be an smooth, connected and non-degenerate projective curve defined over an algebraically closed field with characteristic $0$. Let
$q$ be a general element of $\tau (X,b)$. Then $r_X(q) >b$.
\end{theorem}

From Theorem \ref{i1} we easily get the following result.

\begin{corollary}\label{i2}
Take $b$ and $X$ as in Theorem \ref{i1}. Then there is a quasi-projective variety $J\subset \sigma _b(X)$ such that $\dim J =
\dim
\sigma _b(X) -1$ and $r_X(q) >b$ for all $q\in J$.
\end{corollary}

Let $X\subset \PP^r$, $r\ge 3$, be an integral and non-degenerate projective curve. $X$ is said to be
\emph{tangentially degenerate} if a general tangent line of $X_{\reg}$ meets $X$ at another point of $X$. H. Kaji proved that
in characteristic 0 a non-degenerate smooth projective curve or a projective curve for which the normalization map $C\to X
\subset \PP^r$ is unramified is not tangentially degenerate (\cite[Theorem 3.1 and Remark 3.8]{k}). M. Bolognesi and G. Pirola
extended this result to curves with toric singularities (\cite{bp}). A. Terracini gave an example of a tangentially degenerate
analytic curve in $\CC^3$ (\cite[page 143]{t}). In positive characteristic there are many examples of non-strange curves,
which are tangentially degenerate (\cite[Examples 4.1 and 4.2]{k}, \cite[\S 5]{eh}, \cite[Example 3]{gv}, \cite[Example at page
137]{le}). See \cite{k1} for further results on this topic.

If for a general $p\in X_{\reg}$ the tangent space $T_pX$  is the $2$-secant variety
of the reduced projective set $(X\cap T_pX)_{\red}$, then a general $q\in \tau (X)$ has $X$-rank $2$
and hence for every integer $b\ge 2$ the $X$-rank of a general element of $\tau (X,b)$ is at most $b$. 

In Question \ref{z1} we exclude the case $r=b(1+\dim X)-1$, because in this case usually the answer would be NO (see e.g. \cite{p} for the case of space curves). Usually NO, but in
a few cases YES, as for instance when $r =2b-1$ and $X$ is a rational normal curve by a theorem of Sylvester's (\cite[Theorem 23]{bgi}, \cite{cs}, \cite[\S 1.3]{ik}, \cite[\S 4]{lt}).

Many thanks to a referee for important remarks and corrections.

\section{The proof of Theorem \ref{i1}}
We work over an algebraically closed field with characteristic $0$. Our main reason to require in Theorem \ref{i1} both
characteristic zero and that the curve is smooth
is that \cite{gm} requires both assumptions. Of course, in positive characteristic and/or for singular curves we also need to
assume (at the very least) that $X$ is not tangentially degenerate (Remark \ref{++1}).

Let $X$ be an integral projective curve defined over an algebraically closed field with characteristic $0$. For any integer
$b\ge 2$ let
$\Zz (X,b)$ denote the set of all zero-dimensional schemes $Z\subset X$ with $\deg (Z)=b$,
$b-1$ connected components and with the degree $2$ connected component, $v$, of $Z$ with $v_{\red}\in X_{\reg}$ (any such $v$ is called a tangent vector of $X_{\reg}$). 

\begin{lemma}\label{t3}
Let $X\subset \PP^r$, $r\ge 4$, be an integral and non-degenerate projective curve. For any $o\in X$ let $X[o]\subset \PP^{r-1}$ denote the closure
of $\ell _o(X\setminus \{o\})$ in $\PP^{r-1}$, where $\ell _o: \PP^r\setminus \{o\}\to \PP^{r-1}$ is the linear projection from $o$.  We have $r_X(q) >2$ for a general $q\in \tau (X)$ if
$X[a]$ is not tangentially degenerate for a general $a \in X$.
\end{lemma}

\begin{proof}
Since $X[a]$ is not tangentially degenerate for a general $a\in X$, $X$ is not tangentially degenerate. Take a general $v\in \Zz (X,2)$ and set $\{p\}:= v_{\red}$ and $L:= \langle v\rangle$. Assume that for a general $o\in L$ there is $S_o\subset X$ with $\sharp (S_o)\le 2$ and $o\in \langle S_o\rangle$. Since $L\cap X$ is finite, we have $\sharp (S_o)=2$ for a general $o\in L$. Write $S_o =\{p_1(o),p_2(o)\}$. Since $X$ is not tangentially degenerate, we have $L\ne \langle S_o\rangle$. For a general $o\in L$ the point $p_1(o)$ is general in $X$. 
Hence $X[p_1(o)]$ is a general inner projection of $X$.
More precisely, for a general $(p,o)$ the pair $(p,p_1(o))$ is general in $X^2$ and in particular $p\ne p_1(o)$. Since $(p,p_1(o))$ is general in $X^2$, $\ell _{p_1(o)}({p})$
is a general point point of $X[p_1(o)]$ and in particular $X[p_1(o)]$ is smooth at $\ell _{p_1(o)}({p})$. By construction $\ell _{p_1(o)}(p_2(o))$ is contained in the tangent line of $X[p_1(o)]$ at $\ell _{p_1(o)}({p})$. Since $\langle S_o
\rangle \ne  L$, we have $\ell _{p_1(o)}(p_2(o)) \ne \ell _{p_1(o)}({p})$ and so $X[p_1(o)]$ is tangentially degenerate, a contradiction.
\end{proof}

\begin{remark}\label{ai5}
Let $Y\subset \PP^m$, $m\ge 4$, be an integral and non-degenerate projective curve. Fix an integer $s$ such that $1\le s \le m-3$
and a general $(p_1,\dots ,p_s)\in Y^s$. Set $V:= \langle \{p_1,\dots ,p_s\}\rangle$. Since $Y$ is non-degenerate, we have
$\dim V = s-1$. The trisecant lemma implies that $Y \cap V = \{p_1,\dots ,p_s\}$ (as schemes) and that the linear projection $\ell _V: \PP^m\setminus V\to \PP^{m-s}$ maps
$Y\setminus Y\cap V$ birationally onto its image.
\end{remark}

\begin{remark}\label{++1}
Assume $r\ge 2b\ge 6$ and fix a general $q\in \tau (X,b)$ and a general $o\in \tau (X)$. We claim that $r_X(q)\le b-2+r_X(o)$.
Indeed, by the definition of join we have $q\in  \langle Z\rangle$, where $Z$ is a general element of $\Zz (X,b)$. Write
$Z = v\cup \{p_1,\dots ,p_{b-2}\}$ with $(v,p_1,\dots ,p_{b-2})$ a general element of $\Zz (X,2)\times X^{b-2}$. Set $V:= \langle Z\rangle$,  $L:= \langle v\rangle$, and $\{p\}:= v_{\red}$. Since $v$ is general in $\Zz(X,2)$, a general element of $\langle v\rangle$
is general in $\tau (X)$ and hence it has rank $r_X(o)$. Thus $q$ has rank at most $r_X(o)+b-2$. Thus Theorem \ref{i1} cannot be extended to
tangentially degenerate curves.
\end{remark}

\begin{lemma}\label{t1}
Let $C\subset \PP^r$ , $r\ge 4$, be a smooth, connected and non-degenerate projective curve. Fix a general $(p_1,p_2)\in C^2$
and let
$v =2p_1$ denote the degree $2$ connected effective divisor of $C$ with $p_1$ as its reduction.
Then $\dim \langle v\cup \{p_2\}\rangle =2$ and $C\cap \langle v\cup \{p_2\}\rangle = v\cup \{p_2\}$ as schemes.
\end{lemma}

\begin{proof}
We have $\dim \langle v\cup \{p_2\}\rangle =2$, because $C$ is non-degenerate. Let $3p_1\subset C$ be the degree $3$ effective connected divisor with $p_1$ as its support.
Since we are in characteristic $0$, a general point of $C$ is not a hyperosculating point and so $\dim \langle 3p_1\rangle =2$ and $\langle 3p_1\rangle \cap C$ does not contains
$p_1$ with multiplicity $>3$. We degenerate $v\cup \{p_2\}$ to the effective divisor $3p_1$ and apply \cite[Theorem 1.9]{gm} to $3p_1$.
\end{proof}

\begin{lemma}\label{t7}
Let $X\subseteq \PP^r$,$r\ge 4$,  be a smooth, irreducible and non-degenerate projective curve.
For any $o\in X$ let $\ell _o: \PP^r\setminus \{o\} \to \PP^{r-1}$ be the linear projection from $o$. Call $X[o]$ the closure
of $\ell _o(X\setminus \{o\})$ in $\PP^{r-1}$. Then:
\begin{enumerate}
\item $X$ is not tangentially degenerate and for a general $o\in X$ the curve $X[o]$ is not tangentially degenerate.
\item Let $L\subset \PP^r$ be the tangent line of $X$ at a general point of $X$. Let $\ell _L: \PP^r\setminus L\to
\PP^{r-2}$ denote the linear projection from $L$. Then $\ell _{L|X\setminus X\cap L}$ is birational onto its image.
\item $r_X(q) >2$ for a general $q\in \tau (X)$.
\end{enumerate}
\end{lemma}

\begin{proof}
$X$ is not tangentially degenerate by \cite[Theorem 3.1]{k}. Fix a general $(o,p)\in X^2$ and set $p'= \ell _o(p)$. We have $\ell _o(T_pX) = T_{p'}X[o]$. The set $\Sigma := X[o]\setminus
\ell _o(X\setminus \{o\})$ is finite (it is a single point, the point $\ell _o(T_oX\setminus \{o\})$, but we only need
that it is finite). Assume the existence of $q\in X[o]$ with $q\ne p'$ and $q\in T_{p'}X[o]$. Since we are in characteristic
zero, $X[o]$ is not strange. Hence for a general $p'\in X[o]$ we may assume that $T_{p'}X[o]\cap \Sigma =\emptyset$. Thus
there is $q'\in X\setminus \{o\}$ with $\ell _o(q') =q$. By construction $\langle \{o,q'\}\cup T_pX\rangle$ is a plane.
Note that $(p,o)$ are general in $X^2$. Thus the existence of $q'$ contradicts Lemma \ref{t1}.

Part (3) follows from the second assertion of part (1) and Lemma \ref{t3}.

Now we prove part (2). Assume that $\ell _{L|X\setminus L\cap X}$ is not birational onto its image and call $x\ge 2$ its
degree. Fix a general $q\in X$. The plane $\langle \{q,L\}\rangle$ contains $x-1$ other points of $X$, contradicting Lemma
\ref{t1}.  
\end{proof}

\begin{lemma}\label{t8}
Fix an integer $b\ge 2$. Let $X\subset \PP^r$, $r\ge 2b$,  be a smooth and connected projective variety. Take a general $Z\in
\Zz (X,b)$ and set $V:= \langle Z\rangle$. Then $\dim V=b-1$ and the linear projection $\ell _V: \PP^r\setminus V\to \PP^{r-b}$
induces a birational map of $X$.
\end{lemma}

\begin{proof}
Since $X$ is non-degenerate, we have $\dim V = b-1$. Write $Z = v\cup \{p_1,\dots ,p_{b-2}\}$ with $v$ connected of degree $2$.
Set $L:= \langle v\rangle$. If $b=2$, the lemma is part (2) of Lemma \ref{t7}. Now assume $b>2$. Let $\ell _L: \PP^r\setminus L\to \PP^{r-2}$ denote the linear
projection from $L$ and
$Y\subset \PP^{r-2}$ the closure of the image of $\ell _L(X\setminus X\cap L)$ in $\PP^{r-2}$. Since $Z$ is general, $p_i\notin L$
for all $i$ and $(\ell _L(p_1),\dots ,\ell _L(p_{b-2}))$ is general in $Y^{b-2}$. Apply Remark \ref{ai5} to $Y$.
\end{proof}

\begin{proof}[Proof of Theorem \ref{i1}:]
Since the case $b=2$ is true by part (3) of Lemma \ref{t7}, we may assume $b>2$ and use induction on $b$. Fix a general $q\in \tau (X,b)$. Since $\dim \sigma _{b-1}(X) < \dim \tau (X,b)$, we have $r_X(q) \ge b$. By the
definition of join we have
$q\in 
\langle Z\rangle$, where
$Z$ is a general element of $\Zz (X,b)$. Write
$Z = v\cup \{p_1,\dots ,p_{b-2}\}$ with $(v,p_1,\dots ,p_{b-2})$ a general element of $\Zz (X,2)\times X^{b-2}$. Set $V:= \langle Z\rangle$,  $L:= \langle v\rangle$, and $\{p\}:= v_{\red}$. Since $v$ is general in $\Zz(X,2)$, a general element of $\langle v\rangle$
is general in $\tau (X)$ and hence it has rank $>2$. Since $p_1,\dots ,p_{b-2}$ are general, we
have
$\dim V =b-1$. Assume that for a general
$q\in V$ there is a finite set
$S_q\subset X$ with
$\sharp (S_q) =b$ and
$q\in
\langle S_q\rangle$. Since $r_X(q) \ge b$, $S_q$ is linearly independent. Set $L:= \langle v\rangle$ and $\{p\}:= v_{\red}$.  Since $X$ is not tangentially degenerate (\cite[Theorem 3.1]{k}) and $p$ is general, we have $(X\cap L)_{\red}= \{p\}$. Since a general osculating space of $X$ is not a hyperosculating one, we have $X\cap L = v$ as schemes. Since $q\in V\cap \langle S_q\rangle$, we have
$\rho:= \dim V\cap \langle S_q\rangle \ge 0$. Let $\ell _V: \PP^r\setminus V\to \PP^{r-b}$ denote the linear projection from $V$. Let $W\subset \PP^{r-b}$ be the closure
of $\ell _V(X\setminus X\cap V)$ in $\PP^{r-b}$.

\quad \emph{Claim 1:} $(V\cap X)_{\red} =\{p,p_1,\dots ,p_{b-2}\}$ and $V\cap X =Z$ (as schemes). 

\quad \emph{Proof of Claim 1:} By Lemma \ref{t8} $\ell _{V|X\setminus X\cap V}$ is birational onto its image. Since 
 $(p_1,\dots ,p_{b-2})$ is general in $X^{b-2}$, we have $L\cap \{p_1,\dots ,p_{b-2}\}=\emptyset$ and $(\ell _V(p_1),\dots ,\ell _V(p_{b-2}))$ is a general element of $W^{b-2}$.
 By Remark \ref{ai5}, the fact that $X\cap L =v$ as schemes  and the birationality of $\ell _{V|X\setminus X\cap V}$ we have $V\cap X =Z$ (as schemes).

\quad \emph{Claim 2:} $V\nsubseteq \langle S_q\rangle$. 

\quad \emph{Proof of Claim 2:} We have $\dim  \langle S_q\rangle =\dim V$ and $\sharp (V\cap X)_{\red} =b-1$ by Claim 1.

By Claim 2 we have $\rho \le b-2$. Any covering family of $\rho$-dimensional linear
subspaces of $V$ has dimension at least $b-1-\rho$. Thus any family $\{S_q\}_{q\in T}$ with $\cup _{q\in T} (V\cap \langle S_q\rangle\})$ covering a dense subset of $V$ has dimension
at least $b-1-\rho$.

\quad (a) Assume $S_q\cap \{p,p_1,\dots ,p_{b-2}\}=\emptyset$.  We have $\dim \langle \ell _V(S_b)\rangle =b-2-\rho$. Taking a ramified covering of the parameter space $T$ we may take $S'_q\subset S_q$
with $\sharp (S'_q) = b-1-\rho$ and $\ell _V(S'_q)$ linearly independent. Thus a general $A\in W^{b-1-\rho}$ has the property that $\langle A\rangle$ contains $1+\rho$
points of $W$ (if $A = \ell _V(S_q)$ for a general $q\in T$, use the set $\ell _V(S_q\setminus S'_q)$). This is false by Remark \ref{ai5}.

\quad (b) Assume $S_q \cap \{p_1,\dots ,p_{b-2}\} \ne \emptyset$, say $p_{b-2}\in S_q$. Since $q\ne p_{b-2}$,  $\langle v\cup \{p_1,\dots ,p_{b-3}\}\rangle \cap \langle S_q\setminus \{p_{b-2}\}\rangle\ne \emptyset$. Fix a general $q'\in \langle v\cup \{p_1,\dots ,p_{b-3}\}\rangle \cap \langle S_q\setminus \{p_{b-2}\}\rangle$.  Since $q$
is general in $V$, $q'$ is general in $\langle v\cup \{p_1,\dots ,p_{b-3}\}\rangle$. The set $S_q\setminus \{p_{b-2}\}$ shows that $r_X(q') \le b-1$, contradicting the inductive assumption.

\quad ({c}) Assume $p\in S_q$. Since $q\ne p$, we have $\rho := \dim V\cap \langle S_q\rangle > 0$. By step (b) we may assume that $S_q \cap \{p_1,\dots ,p_{b-2}\} =\emptyset$.
Hence $\ell _V(S_q\setminus \{p\})$ is a well-defined subset of $W$ with cardinality $b-1$ and spanning a linear subspace of dimension $b-2-\rho$. We take $A_q\subset S_q\setminus \{p\}$ with cardinality $b-1-\rho$ and with $\langle \ell _V(A_q)\rangle =\langle \ell _V(S_q\setminus \{p\})\rangle$. As in step (a) we may view $\ell _V(A_q)$ as a general
subset of $W$ with cardinality $b-1-\rho$, while its linear span contains $\ell _V(S_q\setminus \{p\})$, contradicting Remark
\ref{ai5}.\end{proof}

\begin{remark}
Assume characteristic $0$. If no integral and non-degenerate curve in $\PP^r$ is tangentially degenerate, then the proofs just
given show
that Theorem \ref{i1} holds also for singular curves.
\end{remark}

\end{document}